\documentclass[12pt]{amsart}
\usepackage{amsmath,amssymb,amsfonts}
\usepackage{tikz-cd}
\usetikzlibrary{matrix,arrows}

\usepackage[colorlinks=true,linkcolor=red,citecolor=blue]{hyperref}
\newtheorem{theorem}{Theorem}[section]
\newtheorem{proposition}{Proposition}[section]
\newtheorem{corollary}{Corollary}[section]

\newtheorem{remark}{Remark}[section]

\newtheorem{definition}{Definition}[section]

\newcommand\myenum[1]

\begin{document}
	\title{Some properties of uaw-Dunford-Pettis operators}
	\author{S. Boumnidel}
	\address{Sanaa Boumnidel, Department of Mathematics, Abdelmalek Essaadi University,
		P.O. Box 745, Larache 92004, Morocco}
	\email{sboumnidel19@gmail.com}
    \address{Noufissa Hafidi,  Department of Mathematics, Moulay Ismail University, School of Sciences,  P.O. BOX 11201,
Zitoune, Meknes, Morocco}
\email{hafidinoufissa@gmail.com}	
		\author{N. Hafidi}
	\begin{abstract}
		
		This paper introduces and investigates novel properties of uaw-Dunford-Pettis operators on Banach spaces, exploring their relationships with other classes of operators. We further define and characterize new property of Banach lattices. Also we introduce and study the weak$^{\star }$ version of uaw-Dunford-Pettis operator, highlighting their relationships with other classes of operators.

	\end{abstract}
	
	\maketitle
	
	\section{Introduction}
	
	The notion of unbounded absolutely weakly convergence was introduced by O. Zabeti in \cite{article.7}. This is the weak version of unbounded norm convergence which was first introduced by V. Troitsky in \cite{article.8}.

	The studies on those recent convergences play a pivotal role in the study of operator theory and Banach lattices particularly functional spaces. Its applications extend to many fields like measure theory and probability, also, it can indirectly support applied fields like optimization, and mathematical physics.

	The class of uaw-Dunford-Pettis was defined in \cite{article.6}. This article introduces and explores previously unstudied properties and provide new useful characterizations of uaw-Dunford-Pettis operator.  Highlighting its relationship with other classes of operators. Furthermore, Throughout this work we introduce and investigate the weak* version of this operator, that  we called the disjoint weak$^{\star}$ Dunford-Pettis operators. We provide several results about this new class of operators, and we  establish new characterization of Banach Lattices.
	
	In addition, a new property of Banach lattices is introduced  and studied, inspired by the definition of uaw-Dunford-Pettis. Furthermore, we study its relationship with other Banach lattices properties.

	Some other interesting results are also shown.

	\section{Preliminaries}

The unbounded absolute weak convergence (for short, uaw-convergence) was first introduced by Zabeti \cite{article.7}. A sequence $(x_{n})$ is unbounded absolutely weakly convergent (uaw-convergent, for short) to a vector $x$ in $E$ iff $(|x_{n}-x| \land u)$ is weakly convergent to zero for every $u\in E^+$; we write $x_{\alpha}\overset{uaw}{\longrightarrow} x$.\\
The unbounded norm convergence (un-convergence, for short) was introduced by V. Troitsky \cite{article.4}. A sequence is said to be unbounded norm convergent (un-convergence, shortly) to a vector $x$ in $E$ iff $\big|\big| |x_{n}-x| \land u\big|\big|$ converges to zero for every $u\in E^+$ ; we write $x_{n}\overset{un}{\longrightarrow} x$.

	A norm bounded subset $A$ of $E$ is said to be:
	\begin{itemize}
		\item Relatively sequentially unbounded norm compact (relatively $\sigma$-un-compact, shortly) if every sequence in $A$ has a un-convergent subsequence, see \cite{article.4}.
		\item  Relatively sequentially unbounded absolutely weakly compact (relatively $\sigma$-uaw-compact, shortly) if every sequence in $A$ has a uaw-convergent subsequence, see \cite{article.6}.
	\end{itemize}
	Recall that an operator $T:E\longrightarrow F$ is said to be:
	\begin{itemize}
    \item Unbounded absolutely weak Dunford-Pettis operator, (uaw-Dunford-Pettis, shortly), if for every norm
bounded sequence $(x_n)$ in $E$, $x_{n}\overset{uaw}{\longrightarrow}0$ implies  $\|T(x_n)\|\rightarrow 0$, see \cite{article.6}.
		
		\item Sequentially unbounded absolutely weakly compact ($\sigma$-uaw-compact, shortly) if $T(B_{E})$ is $\sigma$-uaw-compact in $F$. Equivalently, for every norm bounded sequence $(x_{n})$ in $E$ its image has a subsequence which is uaw-convergent, see \cite{article.6}.
		\item Sequentially unbounded norm-compact ($\sigma$-un-compact, shortly) if $T(B_{E})$ is $\sigma$-un-compact in $E$. Equivalently, for every norm bounded sequence $(x_{n})$ in $E$ its image has a subsequence, which is un-convergent, see \cite{article.4}.

	\item	Unbounded continuous if for
	each bounded uaw-null sequence $(x_{n})$ in $E$, $(T(x_{n}))$ is weakly null, see \cite{article.usp}
	
		\item $\sigma$-un-continuous, for each norm bounded sequence $(x_{n})$ in $E$, if 
		$x_{n} \overset{un}{\longrightarrow}0$ implies $T(x_{n}) \overset{un}{\longrightarrow}0$, see \cite{article.usp}.
			\item $\sigma$-uaw-continuous for each norm bounded sequence $(x_{n})$ in $E$, if 
			$x_{n} \overset{uaw}{\longrightarrow}0$ implies $T(x_{n}) \overset{uaw}{\longrightarrow}0$, see \cite{article.usp}.
	\end{itemize}
Let also recall that:
	\begin{itemize}
	\item	A Banach space $X$ has the Dunford-Pettis$^{\star}$ property (DP$^{\star}$ property for short),
	if $x_{n} \overset{w}{\longrightarrow} 0$ in $X$ and 
	for every w$^{\star}$-null sequence $(f_{n})$ in $E^{\prime}$, $f_{n}(x_{n}){\longrightarrow} 0$.
	
	\item A Banach lattice $E$ has the uaw-w$^{\star}$DP property, if $\lim\limits_{n\rightarrow \infty} f_{n}(x_{n})= 0$,  for every uaw-null norm bounded sequence $(x_{n})$ in $E$ and every weak$^{\star}$ null sequence $(f_{n})$ in $E^{\prime}$ \cite{article.FF}.
	
		\end{itemize}

	During this article, the following fact will be used:
		Let $(X, ||.||)$ be a normed space. Then $x_{n}\overset{||.||}{\longrightarrow}0$ iff for any subsequence $(x_{n_{k}})$ there is a further subsequence, which we shall denote by $(x_{n_{k}})$ again, such that $(x_{n_{k}})\overset{||.||}{\longrightarrow}0$.
        
		We will use the term operator $T:E\longrightarrow F$ between two Banach lattices to mean a bounded linear mapping.	
		
	For all unexplained terminology and standard facts on vector and Banach lattices, we refer the reader to the monographs \cite{book.1} and \cite{book.2}.
	\section{Main Results}
	
In [\cite{article.6}, Proposition 2.32], the following result was proved  	
	\begin{proposition}
		Let $E$ and $F$ be two Banach lattices.	
		
		If $T:E\longrightarrow F$ is uaw-Dunford-Pettis and $S:G\longrightarrow E$ is $\sigma$-uaw-compact then  $T\circ S$ is compact.

	\end{proposition}

	\begin{remark}\label{nn}
			The converse is not true in general, i.e. If $S$ is a $\sigma$-uaw-compact operator from an arbitrary Banach lattice $G$ to $E$ and $T$ is an operator from $E$ into $F$. The composition $T\circ S$ is compact does not ensure that $T$ is uaw-Dunford-Pettis. Indeed, consider the following: $E=\ell^{\infty}$, $ F= \ell^{\infty}$, $G=\ell^{1}$ and $T = Id_{\ell^{\infty}}$ and the operator $S: \ell^{1} \enskip \longrightarrow  \ell^{\infty}$  defined by 
		
		\begin{center}	
			$S( \alpha_{n} )_{n} \enskip =(\sum_{n=1}^{\infty} \alpha_{n}, \sum_{n=1}^{\infty} \alpha_{n},........)$
		\end{center}
			$S$ is compact  (rank-one oerator). Hence, $Id_{\ell^{\infty}} \circ S= S$ is compact. However, $Id_{\ell^{\infty}}$ is not uaw-Dunford-Pettis. Indeed, consider the unit basis $(e_{n})$ in $\ell^{\infty}$, $e_{n}\overset{||.||}\nrightarrow0$.

	\end{remark}	
	
 By studying the converse, we give the following important characterization:	
	\begin{theorem}\label{Ti2}
		Let $E$ be an atomic KB-space Banach, $Y$ be a Banach space and $T:E\longrightarrow Y$ be an unbounded continuous operator.

		Then the following assertions are equivalent:
		\begin{enumerate}
			\item $T$ is uaw-Dunford-Pettis.
			\item $T$ carries relatively $\sigma$-uaw-compact subsets of $E$ into norm totally bounded subsets of $F$.
			\item For an arbitrary operator $S$ from a Banach lattice $G$ into $E$, $T\circ S$ is compact.
			\item For an arbitrary operator $S$ from $\ell^{1}$ into $E$, $T\circ S$ is compact.
		\end{enumerate}
	\end{theorem}
	\begin{proof}
		$(1)\Longrightarrow (2)$ Let $W$ be a $\sigma$-uaw-compact subset of $E$. If
		$(x_{n})$ is a sequence of $W$, then there exists a subsequence $(y_{n})$ of $(x_{n})$ satisfying $ y_{n}\overset{uaw}{\longrightarrow}y $, in $Y$ and so $T(y_{n}) - T(y) \overset{||.||}{\longrightarrow} 0$, which proves that
		$T(W)$ is a norm totally bounded subset.	
		
		$(2)\Longrightarrow (3)$ and $(3)\Longrightarrow (4)$ are obvious.

		$(4)\Longrightarrow (1)$ Let $x_{n}\overset{uaw}{\longrightarrow}0$, and define the operator $S: \ell^{1} \enskip \longrightarrow  E$  defined by  \begin{center}
			$S((\lambda_{k})_{k=1}^{\infty })=\sum\limits_{k=1}^{\infty }\lambda_{k}x_{k}$ \hspace{1cm} for each $(\lambda_{k})_{k=1}^{\infty }\in \ell^{1}$.
		\end{center}
		Since $E$ is an atomic KB-space then, it follows from [\cite{article.4}, Theorem 7.5] that there exists a subsequence $(x_{n_{k}})$ of $(x_{n})$  such that $x_{n_{k}}\overset{un}{\longrightarrow} y$ for some $y$ in $E$, hence $x_{n_{k}}\overset{uaw}{\longrightarrow} y$. Consequently, the identity operator $Id_{E}$ is $\sigma$-uaw-compact. By virtue of [\cite{article.6}, Proposition 2.32], $Id_{E}\circ S = S$ is $\sigma$-uaw-compact.

		Now, according to our hypothesis $T\circ S$ is compact. Let $(e_{n})$ denotes the sequence of basic vectors of $\ell^{1}$. On one hand, $T\circ S(e_{n})=T(x_{n})\overset{w}{\longrightarrow}0$ holds because $T$ is unbounded continuous. On the other hand, since
		$T\circ S$ is compact, we see that every subsequence of $(T\circ S(e_{n}))$ has a subsequence
		converging in norm to zero. Therefore, $T(x_{n})\overset{||.||}{\longrightarrow}0$ holds, and this proves that $T$ is a uaw-Dunford-Pettis operator.
	\end{proof}

	\begin{remark}
		\begin{enumerate}	
			
			\item In Theorem \ref{Ti2} the condition of unbounded continuity of $T$ can not be removed, let consider a Banach lattice $F\neq \{0\}$ and an element $y$ in $F$ with $y\neq 0$, consider also $S=  Id_{\ell^{1}}$ and  $T: \ell^{1} \enskip \longrightarrow  F$  defined  as below, 
			
			\begin{center}
				$T((\lambda_{k})_{k=1}^{\infty })=\sum\limits_{k=1}^{\infty }\lambda_{k}y$ \hspace{1cm} for each $(\lambda_{k})_{k=1}^{\infty }\in \ell^{1}$.
			\end{center}
			By [\cite{article.4}, Theorem 7.5] $S$ is $\sigma$-uaw-compact, but $T$ is not unbounded continuous. Indeed let consider the unit basis sequence $(e_{n})$ of $\ell^{1}$, we know that  $e_{n}\overset{uaw}{\longrightarrow}0$, however $T(e_{n})= y\overset{w}{\nrightarrow}0$.

			Finally; $T \circ Id_{\ell^{1}}= T$ is compact but $T$ is not uaw-Dunford-Pettis. 
			
			\item In Theorem \ref{Ti2}, the condition $E$ is KB-space can not be dropped, indeed, let consider the example in Remark \ref{nn}, $E=\ell^{\infty}$, $\ell^{\infty}$ is not KB-space and $Id_{\ell^{\infty}}$ is unbounded continuous (by [\cite{article.7}, Theorem 7]), but $Id_{\ell^{\infty}}$ is not uaw-Dunford-Pettis.

		\end{enumerate}			
	\end{remark}

We have the following result:

		\begin{theorem}\label{t2}
		Let $E$ be a Banach lattice and $Y$ be a Banach space. If each weak Dunford-Pettis operator $T:E\rightarrow Y$ is uaw-Dunford-Pettis, then one of the following assertion is valid:
		\begin{enumerate}
			\item The norm of $E^{\prime}$ is order continuous,
			\item $Y$ has the Schur property.
		\end{enumerate}
	\end{theorem}
	\begin{proof}

		It suffices to establish that if the norm of $E^{\prime}$ is not order continuous, then $Y$ has the Schur property.

		Since the norm of $E^{\prime }$ is not order continuous, by  [\cite{book.2}, Theorem 2.4.14] we may assume that $\ell^{1}$ is a closed sublattice of $E$, then it follows from  [\cite{book.2}, Proposition 2.3.11] that there is a positive projection $P$ from $E$ onto $\ell^{1}$.

		We will prove that $Y$ has the Schur property. To this end, let $(y_{n})$ be a weakly null sequence in $Y$ and consider the operator $T=S\circ P:E\longrightarrow \ell^{1}\longrightarrow Y$, where $S$ is the operator defined by

		$$    
		\begin{array}{lrcl} 
		S : & \ell^1 & \longrightarrow &  Y \\
		& (\alpha _{n}) & \longmapsto & \ \displaystyle\sum_{n=1}^{\infty}\alpha_{n}y_{n}  \end{array}  
		$$
		
		 $T$ is a weak Dunford-Pettis  operator (because $\ell^{1}$ has the Dunford-Pettis property), and by hypothesis $T$ is uaw-Dunford-Pettis. Now, let $(e_{n})$ be the standard basis of $\ell^{1}$, we have $e_{n} \overset{uaw}{\longrightarrow} 0$, then $\|T (e_{n})\|=\|y_{n}\|\rightarrow 0$. Therefore, $Y$ has the Schur property.
	\end{proof}

Obviously, every uaw-Dunford-Pettis is $\sigma$-un-continuous. But the converse is not true in general. Indeed, $Id_{\ell^{\infty}}:\ell^{\infty}\longrightarrow \ell^{\infty}$ is $\sigma$-un-continuous which is not uaw-Dunford-Pettis. 

For this converse, we have the following.
		
	\begin{proposition}
		Let $E$ and $F$ be a pair of Banach lattices, such that $F$ has strong unit. If $E$ is order continuous, then, each $\sigma$-un-continuous $T:E\rightarrow F$ is uaw-Dunford-Pettis
	\end{proposition}
	
	\begin{proof}
		Let  $(x_{n})$ be a sequence in $E$ such that $x_{n}\overset{uaw}\longrightarrow 0$, by [\cite{article.7}, Theorem 4], $x_{n}\overset{un}\longrightarrow 0$, then $T(x_{n})\overset{un}\longrightarrow 0$. Using [\cite{article.4}, Theorem 3.2],  we obtain $T(x_{n})\overset{||.||}\longrightarrow 0$, as desired.
	\end{proof}

Next, we introduce a new property in Banach lattices.
	\begin{definition}\label{kkkk}
		A Banach lattice $E$ has the weak$^{\star}$ disjoint Dunford-Pettis property if $\lim\limits_{n\rightarrow \infty} f_{n}(x_{n})= 0$, for every disjoint norm bounded sequence $(x_{n})$ in $E$ and every weak$^{\star}$ null sequence $(f_{n})$ in $E^{\prime}$.
		
	\end{definition}

	\begin{remark}\label{LOU}	
	By Proposition \ref{cou}, $\ell^{\infty}$ has the disjoint weak$^{\star}$ Dunford-Pettis property.
				
		\end{remark}
	\begin{remark}
	 $\ell^{1}$ fails the disjoint weak$^{\star}$ Dunford-Pettis property.

	\end{remark}
	\begin{proof}
		$(\ell^{1})^{\prime}$ is not order continuous. Then it follows from [\cite{zan} Theorem 116.1] that, there exists a norm bounded disjoint sequence $(u_{n})$ of positive elements in $E$ which does not converge weakly to zero. Without loss of generality, we may assume that $\| u_{n}\|\leq 1$ for any $n$. And there exists $\varepsilon>0$ and $0\leq \phi \in E^{\prime}$ such that $\phi (u_{n})>\varepsilon$ for all $n$. Then by \cite[Theorem 116.3]{zan}, we know that the components $\phi_{n}$ of $\phi$ in the carriers $C_{u_{n}}$ form an order bounded disjoint sequence in $(E^{\prime})_{+}$ such that $\phi_{n}(u_{n})=\phi(u_{n})$ and $\phi_{n}(u_{m})=0$ for each $n,m$ $\in \mathbb{N}$. Since $\ell^{1}$ is order continuous then by \cite[Corollary 2.4.3]{book.2}, $\phi_{n}\overset{w*}{\longrightarrow} 0$. However, $\phi_{n}(u_{n}) \nrightarrow 0$. Hence $\ell^{1}$ fails the disjoint weak$^{\star}$ Dunford-Pettis property.
		
	\end{proof}
	
	By virtue of the disjoint weak$^{\star}$ Dunford-Pettis property we give this important characterization.
	\begin{theorem}\label{mou}
		
		Let $E$ be a $\sigma$-Dedekind complete Banach lattice, the following assertions are equivalent:

		(1) $E$ has the disjoint weak$^{\star}$ Dunford-Pettis property.

		(2) Every operator $T : E \longrightarrow c_{0}$ is  M-weakly operator.

		(3) Every operator $T : E \longrightarrow c_{0}$ is uaw-Dunford-Pettis operator.

	\end{theorem}	
	
	\begin{proof}	
		$(1)\Longrightarrow (2)$ Let $T : E \longrightarrow c_{0}$ be an arbitrary linear operator. We will prove that $||T(x_{n})||\longrightarrow 0$, for every norm bounded disjoint sequence
		$(x_{n}) \subset E$. Assume by way of contradiction that $||( T(x_{n}) )||$ does not converge to 0 for
		some disjoint sequence $(x_{n}) \in E$. Then, we can suppose that there would exists some $\epsilon > 0$ such that  $||T(x_{n})|| > \epsilon$
		for all $n \in  \mathbb{N}$.

		Using the same technical proof used in  [\cite{article.SDP}, Proposition 2.1] there exists a canonical projection, denoted $\pi_{k_{n}}$, from
		$c_{0}$ into its coordinate space  $\mathbb{R}$ such that $||T(x_{n})|| = |\pi_{k_{n}} \circ
		T(x_{n})|$. By the proof of [\cite{article.SDP}, Proposition 2.1]  the sequence $(k_{n}) \subset \mathbb{N}$
		can not be bounded. Again by passing to a subsequence if necessary, we can suppose
		that $({k_{n}})$ is strictly increasing. Then $(\pi_{k_{n}} \circ T)$ is a w$^{\star}$-null sequence of $E^{\prime}$.

		Note that $({x_{n}})$ is a disjoint sequence of $E$. then, by virtue of [\cite{book.1}, Ex.22, p.77]
		there exists a disjoint sequence $(f_{n})$ in $E^{\prime}$ such that
		\begin{center} 
			$|f_{n}| \leq |\pi_{k_{n}} \circ T|$

			$f_{n}(x_{n}) = (\pi_{k_{n}} \circ T )(x_{n}) = \pi_{k_{n}}
			(T(x_{n}))$.
		\end{center} 
		Since $(\pi_{k_{n}} \circ T)$ is w$^{\star}$-null sequence, by  [\cite{article.JJ}, Lemma 2.2] we have $(f_{n})$ is w$^{\star}$-null sequence in $E^{\prime}$. By hypothesis, it follows from Definition \ref{kkkk} that $f_{n} (x_{n})\longrightarrow 0$ as $n \longrightarrow \infty$.
		On the other hand,

		\begin{center} 
			$| f_{n} (x_{n})| = |(\pi_{k_{n}} \circ T )(x_{n})| = |\pi_{k_{n}}
			(T(x_{n}))| = T(x_{n}) > \epsilon > 0$. \end{center} 
		This is a contradiction. Hence, $||T(x_{n})|| \longrightarrow 0$ for every disjoint
		sequence $(x_{n})$ of $E$, that is, $T$ is an M-weakly compact operator.
		
		$(2)\Longrightarrow (1)$ By Definition \ref{kkkk}, to prove that $E$ has the disjoint weak$^{\star}$Dunford-Pettis property, we need to prove that for each norm bounded disjoint sequence $(x_{n})$ of $E$ and each
		disjoint w$^{\star}$-null sequence $(f_{n}) \subset E^{\prime}$, we have $f_{n} (x_{n})\longrightarrow 0$. Let consider the operator
		$T : E \longrightarrow c_{0}$ defined by $T(x) = (f_{n}(x))_{n}$ for any $x \in E$. By hypothesis, $T$ is M-weakly compact. Therefore, $||T(x_{n})|| \longrightarrow 0$ as $n \longrightarrow \infty$, hence $f_{n} (x_{n})\longrightarrow 0$, as desired.
		
		$(2) \Leftrightarrow (3)$ By virtue of [\cite{article.usp}, Theorem 18].
		
	\end{proof}

	\begin{proposition}\label{gou}
		Let $E$ be a Banach lattice with the disjoint weak$^{\star}$ Dunford-Pettis$^{\star}$ property, If $E^{\prime}$ and $E$ are order continuous, then $E$ is finite dimensional. 
		
	\end{proposition}
	\begin{proof}
		
		Let $(x_{n})$ be a disjoint norm bounded sequence of $E$, by Lemma 2 and Lemma 1 of \cite{article.7}, $|x_{n}| \overset{uaw}{\longrightarrow} 0$, then by [\cite{article.7}, Theorem 7] we have $ |x_{n}| \overset{w}{\longrightarrow} 0$. ($^{\star}$)

		Let $(f_{n})$ be a norm bounded disjoint sequence in $E^{\prime}$, by \cite[Corollary 2.4.3]{book.2}, $f_{n}$ is w$^{\star}$-null sequence, by assumption, $f_{n}(x_{n}){\longrightarrow} 0$. ($^{\star}$$^{\star}$)

		From ($^{\star}$) and ($^{\star}$$^{\star}$) and [\cite{article.DF}, Corollary 2.6], we conclude that  $x_{n} \overset{||.||}{\longrightarrow} 0$.

		 Now, by [\cite{article.usp}, Theorem 18] and [\cite{article.dim}, Corollary 5] we claim that $E$ is finite dimensional.	
	\end{proof}
	We introduce the weak$^{\star}$ version of uaw-Dunford-Pettis operators.
	\begin{definition}
		An operator $T:E\longrightarrow Y$ from a Banach lattice 
 $E$ into a Banach space $Y$ is said to be disjoint weak$^{\star}$ Dunford-Pettis, if for every norm bounded disjoint sequence $(x_n)$ in $E$, and  every  sequence $f_n\overset{w^{\star} }{\longrightarrow} 0$ in $Y^{\prime}$ imply, $\lim\limits_{n\rightarrow \infty } f_{n}(T(x_{n}))=0$.
\end{definition}

	Now, we consider some ideal properties for the class of disjoint weak$^{\star}$ Dunford-Pettis operator. But before investigating this, we give  the following Example.

$$    
\begin{array}{lrcl} 
 T : & \ell^1 & \longrightarrow & \ell^{\infty} \\
    & \lambda_n & \longmapsto & \ ( \displaystyle\sum_{n=1}^{\infty}\lambda_n, \displaystyle\sum_{n=1}^{\infty}\lambda_n,...)  \end{array}  
$$

  By Remark \ref{LOU} and Proposition \ref{pr333} $Id_{\ell^{\infty}}$ is a disjoint weak$^{\star}$ Dunford-Pettis operator, However  $T=Id_{\ell^{\infty}}\circ T$ is not disjoint weak$^{\star}$ Dunford-Pettis. Indeed, let $f_n\overset{w^{\star} }{\longrightarrow} 0$ in $(\ell^{\infty})^{\prime}$, then there exists a unique sequence of scalars  $(\alpha_n)$  such that $f_n= \displaystyle\sum_{n=1}^{\infty}\alpha_n e_n$. Moreover, we have that

\begin{align*}
\langle T(e_n), f_n \rangle & = \langle ( \sum_{n=1}^{\infty}e_n, \sum_{n=1}^{\infty}e_n,.,...), \sum_{n=1}^{\infty}\alpha_n e_n \rangle \\	& = \sum_{n=1}^{\infty}\alpha_n \langle ( \sum_{n=1}^{\infty}e_n, \sum_{n=1}^{\infty}e_n,..), e_n \rangle	\\ 	& = \sum_{n=1}^{\infty}\alpha_n . \sum_{n=1}^{\infty}e_n \langle  (1,1,..), e_n \rangle		 \\ & =  \sum_{n=1}^{\infty}\alpha_n .\sum_{n=1}^{\infty}e_n  \nrightarrow 0
\end{align*}

   We clearly see that $T$ is not a disjoint weak$^{\star}$ Dunford-Pettis operator. We also claim that $T$ is not  disjointness preserving operator. Indeed, since $(e_n)$ is disjoint then  $T(e_n\land e_m)$ is null in $\ell^{\infty}$, but
 \begin{center} $T(e_n)\land T(e_m) = (1,1,1,....) \neq (0,0,0,.....)$

	\end{center}
	
	\begin{proposition}\label{ll}
		Let $S:E\longrightarrow F$ and $T:F\longrightarrow G$ be two operators between Banach
		lattices $E$, $F$ and $G$.
		\begin{enumerate}
			\item If $T$ is bounded and $S$ is a disjoint weak$^{\star}$ Dunford-Pettis operator, then $T\circ S$ is also disjoint weak$^{\star}$ Dunford-Pettis.
			\item If $S$ is disjointness preserving and $T$ is a disjoint weak$^{\star}$ Dunford-Pettis operator, then $T\circ S$ is disjoint weak$^{\star}$ Dunford-Pettis.
		\end{enumerate}
	\end{proposition}
	\begin{proof}
		$(1)$ Let $(x_{n})$ be a norm bounded disjoint sequence of $E$ and $(f_{n})\subset G^{\prime}$ such that  $f_{n}\overset{w^{\star}}{\longrightarrow}0$. Then $f_{n}\circ T\overset{w^{\star}}{\longrightarrow}0\quad(\star)$.

		By assumption $S$ is disjoint weak$^{\star}$ Dunford-Pettis, it follows from $(\star)$ that $f_{n}\circ T\left(S(x_{n})\right)=f_{n}\left(T\circ S(x_{n})\right)\longrightarrow 0$, for each null weak$^{\star}$ convergent sequence $(f_{n})$ in $G^{\prime}$. Hence $T\circ S$ is disjoint weak$^{\star}$ Dunford-Pettis.

		$(2)$ Let $(x_{n})$ be a disjoint sequence of $E$  and $(f_{n})\subset G^{\prime}$ such that $f_{n}\overset{w^{\star}}{\longrightarrow}0$. As, $S$ is disjointness preserving operator, it follows that $(S(x_{n}))$ is a disjoint sequence.  In this step, assumption asserts that
		$f_{n} \left(T(S(x_{n})\right)\longrightarrow 0$. Therefore, $T\circ S$ is disjoint weak$^{\star}$ Dunford-Pettis.

	\end{proof}	
	\begin{proposition}
		Every M-weakly compact $T:E\longrightarrow Y$ is disjoint weak$^{\star}$ Dunford-Pettis. 
	\end{proposition}
	\begin{proof}
Let $T:E\longrightarrow Y$ be an M-weakly compact, by the inequality \begin{center} $|f_{n}(T x_{n})|\leq (\sup_{n}\|f_{n} \|) \| T(x_{n})\|$\end{center} We conclude that $T$ is disjoint weak$^{\star}$ Dunford-Pettis.

	\end{proof}
	\begin{remark}
The converse  of the above result is not true in general. In fact, let consider the identity operator on $\ell^{\infty}$.  By Remark \ref{LOU} and Proposition \ref{pr333}, we know that $Id_{\ell^{\infty}}$ is disjoint weak$^{\star}$ Dunford-Pettis but it is not M-weakly compact.
\end{remark}
	\begin{proposition}\label{pr3}
	Let $E$ and $F$ be two Banach lattices, if $E^{\prime}$ and $F$ are order continuous. Then every positive disjoint weak$^{\star}$ Dunford-Pettis $T : E \longrightarrow F$ is M-weakly compact (hence uaw-Dunford-Pettis).

\end{proposition}
\begin{proof}
	
Let $T : E \longrightarrow F$ be a positive disjoint weak$^{\star}$ Dunford-Pettis operator and  let $(x_{n})$ be a disjoint sequence in $B_{E}$. By [\cite{article.DF}, Corollary 2.6], it suffices to prove that $| Tx_{n}| \overset{w}\longrightarrow 0$ in $E$ and $f_{n}(Tx_{n}) \longrightarrow 0$ for every disjoint  norm bounded sequence $(f_{n}) \subset (F^{\prime})^{+}$.\\
Let $(x_{n})$ be a disjoint sequence in $E$, by Lemma 2 of \cite{article.7}, $x_{n}\overset{uaw}\longrightarrow0$. Then, by Lemma 1 of \cite{article.7}, $|x_{n}|\overset{uaw}\longrightarrow0$. Now, using [\cite{article.7},Theorem 7] it follows that $x_{n}\overset{w}\longrightarrow0$ and $|x_{n}|\overset{w}\longrightarrow0$.

In this step, we infer from [\cite{book.1}, Theorem 1.23] that for each $n\in\mathbb{N}$
there exists some $g_{n} \in [-f,f]$ with $f|T(x_{n})| = g_{n}(T(x_{n}))$. Since $|x_{n}|\overset{w}\longrightarrow0$ and $T$ is positive, we will have
\begin{align*}
0 \leq f|Tx_{n}| = g_{n}(Tx_{n})  & = T^{\prime}(g_{n})(x_{n})\\
        & \leq |T^{\prime}(g_{n})||x_{n}| \\
        &\leq T^{\prime}(f)|x_{n}|\longrightarrow0
\end{align*}

Therefore, $|T(x_{n})|\overset{w}\longrightarrow0$.

Now, let $(f_{n})\subset (F^{\prime})^{+}$ be a disjoint and norm bounded sequence. By \cite[Corollary 2.4.3]{book.2}, $f_{n}\overset{w^{*}}\rightarrow 0$. As $T$ is disjoint weak$^{\star}$ Dunford-Pettis, it follows that $f_{n} (T (x_{n}))\longrightarrow 0$. This proves that $T$ is M-weakly compact. (By [\cite{article.usp}, Theorem 18], $T$ is uaw-Dunford-Pettis).
	
\end{proof}

The next result provides important characterization of the the disjoint weak$^{\star}$ Dunford-Pettis property.
	\begin{proposition}\label{pr333}
	Let $E$ be a Banach lattice. Then the following statements are equivalent:
	
	\begin{enumerate}
		\item $E$ has the disjoint weak$^{\star}$ Dunford-Pettis property.
		
		\item Each operator $T:E\longrightarrow E$ is disjoint weak$^{\star}$ Dunford-Pettis.
		
		\item The identity operator of $E$ is disjoint weak$^{\star}$ Dunford-Pettis.
		\item	For every norm bounded disjoint sequence $(x_{n})$ in $E$, the subset $\{{x_{n},n\in N\}}$ is limited.

	\end{enumerate}
\end{proposition}
\begin{proof}
	$(1)\Longrightarrow(2)$ Let $(f_{n}) \subset E^{\prime}$ such that $f_{n}\overset{w^{\star}}{\longrightarrow} 0$, and let consider a norm bounded disjoint sequence $(x_{n})\subset E$.\\
	Obviously, $f_{n}\circ T\overset{w^{\star}}{\longrightarrow} 0$, by assumption, $f_{n} (T (x_{n}))\longrightarrow 0$, this ends the proof.
	
	$(2)\Longrightarrow (3)$ Obvious.

	$(3)\Longrightarrow (4)$ By way of contradiction, let suppose that $A= \{{x_{n},n\in N\}}$ is not limited. Let $(f_{n})$ be a weak$^{\star}$ null sequence of $E^{\prime}$ and let suppose that $(f_{n})$ does not converge uniformly on $A$. Then there exists a subsequence $(x_{n_k})$ of $(x_{n})$ such that 	
	$0<\varepsilon<|f_{n}(x_{n_k})|$.
	
	By assumption, $\lim\limits_{n\rightarrow \infty } f_{n}(x_{n_k})=0$. This is a contradiction. Then $A$ is limited.

	$(4)\Longrightarrow (1)$	By the virtue of the inequality $|f_{n}(x_{n})|\leq\sup_{k}|f_{n}(x_{k})|$ for every weak$^{\star}$ null sequence $(f_{n})$ of $E^{\prime}$.
\end{proof}

	The following result follows from the Proposition \ref{pr333}.
	\begin{corollary}\label{corr}
		Let $E$ be a Banach lattice. Then the following statements are equivalent:
		
		\begin{enumerate}
			\item $E$ has the disjoint weak$^{\star}$ Dunford-Pettis property.
			
			\item For an arbitrary Banach space $Y$, every operator $T:E\longrightarrow Y$ is disjoint weak$^{\star}$ Dunford-Pettis.
			
			\item For an arbitrary Banach lattice $F$, every operator $T:E\longrightarrow F$ is disjoint weak$^{\star}$ Dunford-Pettis.
		\end{enumerate}
			\end{corollary}
	\begin{corollary}
	Let $E$ be a Banach lattice, if $E$ is order continuous, then every positive operator	$T:\ell^{\infty}\longrightarrow E$ is uaw-Dunford-Pettis.
\end{corollary}	
	\begin{proof}	
	Let $T:\ell^{\infty}\longrightarrow E$ be an arbitrary positive operator. It follows from Remark \ref{LOU} and Corollary \ref{corr} that $T$ is disjoint weak$^{\star}$ Dunford-Pettis, moreover by Proposition \ref{pr3}, $T$ is uaw-Dunford-Pettis. 
	\end{proof}	
	
	We discuss below the relationships between the disjoint weak$^{\star}$ Dunford-Pettis property and other defined properties.

	\begin{proposition}\label{cou}
		Let $E$ be a Banach lattice with the Dunford-Pettis$^{\star}$ property. If $E^{\prime}$ is order continuous, then $E$ has the disjoint weak$^{\star}$ Dunford-Pettis property. 
		
	\end{proposition}
	\begin{proof}
		
		Let $(x_{n})$ be a disjoint norm bounded sequence of $E$, by Lemma 2 and Theorem 7 of \cite{article.7}, $x_{n} \overset{w}{\longrightarrow} 0$, then for every w$^{\star}$-null sequence $(f_{n})$ in $E^{\prime}$, we have $f_{n}(x_{n}){\longrightarrow} 0$. This ends the proof.

	\end{proof}
	\begin{remark}
		
		The condition $E^{\prime}$ is order continuous in Proposition \ref{cou} can not be dropped. Indeed, $\ell^{1}$ has the Dunford-Pettis$^{\star}$ property but fails the disjoint weak$^{\star}$ Dunford-Pettis property.	
		
	\end{remark}	
	\begin{proposition}\label{lou}

		For a $\sigma$-Dedekind complete Banach lattice $E$. The disjoint weak$^{\star}$ Dunford-Pettis property and the uaw-w$^{\star}$DP property properties are equivalent.	
		
	\end{proposition}	
	
	\begin{proof}
		
		Suppose that $E$ has the disjoint weak$^{\star}$ Dunford-Pettis property, and let consider the operator $T : E \longrightarrow c_{0}$ defined by $T(x) = (f_{n}(x))_{n}$. By Theorem \ref{mou}, $T$ is uaw-Dunford-Pettis, then for every norm bounded $(x_{n})$ in $E$ such that $x_{n} \overset{uaw}{\longrightarrow} 0$, and every w$^{\star}$-null sequence $(f_{n})$ in $E^{\prime}$, $f_{n}(x_{n}){\longrightarrow} 0$. Hence $E$ has the uaw-w$^{\star}$DP property.

		For the converse, assume  that $E$ has the uaw-w$^{\star}$DP property and let $(x_{n})$ in $E$ be a norm bounded disjoint, by [\cite{article.7}, Lemma 2]  $x_{n} \overset{uaw}{\longrightarrow} 0$, then for every w$^{\star}$-null sequence $(f_{n})$ in $E^{\prime}$, $f_{n}(x_{n}){\longrightarrow} 0$ which ends the proof.
		
	\end{proof}		
	
		In the next result, we establish the relationship between the class of disjoint weak$^{\star}$ Dunford-Pettis operators and the L-weakly compact one.

	\begin{proposition}
	Let $E$ and $F$ a pair of Banach lattices such that $E'$ is order continuous. Each positive L-weakly compact operator $T:E\rightarrow F$ is disjoint weak$^{\star}$ Dunford-Pettis.		
\end{proposition}

\begin{proof}	Assume by way of a contradiction that $T$ is not disjoint weak$^{\star}$ Dunford-Pettis, that is, there exist $(x_n)$ a bounded disjoint sequence in $E$ and $(f_{n})$ a sequence of $F'$ such that $f_n \overset{w^{\star}}{\longrightarrow} 0$ in $F'$ with  $f_{n}(Tx_{n})\nrightarrow 0$. Therefore there exists a subsequence of $(f_{n}(Tx_{n}))$ (which we can denote by $(f_{n}(Tx_{n}))$) and $\varepsilon>$ such that  $|f_{n}(Tx_{n})|>\varepsilon$, for all $n$. In particular, we infer from the inequality  $|f_{n}(Tx_{n})| \leq |f_{n}|(T|x_{n}|)$, that  $|f_{n}|(T|x_{n}|)|>\varepsilon$. According to  [\cite{article.7},Theorem 7]  $|x_n|\overset{w}{\longrightarrow} 0$, and hence $T(|x_n|)\overset{w}{\longrightarrow} 0$. Now, an easy inductive argument shows that there exists a subsequence $(z_n)$ of $(|x_n|)$ and subsequence $(g_n)$ of $(f_n)$ such that $$ |g_{n}|(T(z_n))|>\varepsilon $$
	
	and $$(4^n\sum_{i=1}^{n}
	|g_i
	|) (T(z_{n+1})) < \frac{1}{n}, \text{ for all $n\geq 1$.}$$
	we consider $h =\sum_{i=1}^{\infty} 2^{-n}
	|g_n|$ and $ h_n=(|g_{n+1}|-4^{n}\sum_{i=1}^{n}|g_i|-2^{-n}h)^+$
	[\cite{book.1}, Lemma 4.35 ] the sequence $(h_n)$ is disjoint. Since $0 \leq h_n \leq |g_{n+1}|$ for all
	$n \geq 1$ and $(g_n)$ is weak* null in $F$, then from [\cite{article.JJ}, Lemma 2.2] $(h_n)$ is weak* null in
	$F$, from the inequality
	
	\begin{align*}
		h_n(T (z_{n+1})) & \geq  |g_{n+1}|-4^{n}\sum_{i=1}^{n}|g_i|-2^{-n}h (T (z_{n+1})) \\	& \geq \varepsilon -\frac{1}{n} -2^{-n} h (T (z_{n+1}))				\end{align*}
	
	we see that $h_n(T (z_{n+1})) \geq \frac{\varepsilon}{2}$ must hold for all $n$ sufficiently large (because $2^{-n}h (T (z_{n+1}))\rightarrow 0$ ). Now, since  $T$ is an L-weakly compact operator, and by using [\cite{book.2}, Proposition 3.6.2] and the inequality  $$|h_{n}(Tz_{n+1})| \leq \sup_{x\in B_{E}}{|h_{n}|(|Tz_{n+1}|)|}= \sup_{x\in B_{E}}{|f_{n}|(Tz_{n+1})|}\longrightarrow 0$$ 
	Which shows a contradiction and finishes the proof.
\end{proof}

Recall that an order continuous Banach lattice $E$ is said to have the subsequence splitting property (SSP, shortly) if for any norm bounded
sequence $(x_n)$ there exist a subsequence $(x_{n_k})$ of $(x_n)$ and two sequences $(y_k)$ and $(z_k)$ such that
$x_{n_k}=y_k+z_k$, $(y_k)$ is L-weakly compact, $(z_k)$ is pairwise disjoint and $y_k \perp z_k$ for all $k$ (see \cite{FLO}).

In general, disjoint weak$^{\star}$ Dunford-Pettis  operator is not necessary  L-weakly compact. Indeed, $Id_{\ell^{\infty}}$ is disjoint weak$^{\star}$ Dunford-Pettis but it is not L-weakly compact.

In the following, we give necessary conditions for which each disjoint weak$^{\star}$ Dunford-Pettis  operator is L-weakly compact.

\begin{proposition}
	Let $E$ and $F$ a pair of Banach lattices, such that $E$ has the subsequence splitting property, and $F$ is order continuous and having the DP$^{\star}$ property. Then, each disjoint weak$^{\star}$ Dunford-Pettis operator $T:E\rightarrow F$ is L-weakly compact.		
\end{proposition}

\begin{proof}
	To show that $T:E\rightarrow F$ is L-weakly compact operator, we need to prove that $T(B_{E})$ is L-weakly compact set, so by  [\cite{book.2}, Proposition 3.6.2] it suffices to show that $$\sup_{x\in B_{E}}{|f_{n}|(|Tx|)|}\rightarrow 0$$ for every norm bounded disjoint sequence  $(f_{n}) \subset (F^{\prime})$. Assume by way of contradiction that $\displaystyle\sup_{x\in B_{E}}{|f_{n}|(|Tx|)|}\nrightarrow 0$  for some disjoint sequence $(f_{n})$ in $F^{\prime}$, that is, there exists $\varepsilon>0$ such that $\sup_{x\in B_{E}}{|f_{n}|(|Tx|)|}> \varepsilon$ for all $n$. Therefore, for each $n$ we choose some $(x_n)\in B_{E}$ and some $(g_{n}) \subset (F^{\prime})$ with $|g_n|\leq |f_{n}|$, and such that $|g_{n}(Tx_n)|> \varepsilon$. Now, since $E$ has the SPP, we can find a subsequence $(x_{n_{k}})$ and terms $(y_{k}),(z_{k})$ such that: $x_{n_{k}}=y_{k}+z_{k}$
	
	$\bullet$	$(y_{k})$ is L-weakly compact sequence, hence relativey weakly compact.
	
	$\bullet$  \((z_{k})\) is a disjoint sequence. 
	
	So, there exists a subsequence of $(y_{k})$, that we denote without loss of generality by $(y_{k})$, and  $y\in E$ such that $y_{k}-y \overset{w}{\rightarrow} 0$. On the other hand, by \cite[Corollary 2.4.3]{book.2}, it follows that $g_{n} \overset{w^{\star}}{\longrightarrow} 0$. By  hypothesis $T$ is disjoint weak$^{\star}$ Dunford-Pettis, hence $|g_{k}(Tz_{k})| \rightarrow 0$. Now, as $F$ has the DP* property it follows that $|g_{k}(T(y_{k}-y ))| \rightarrow 0$, and we have also $|g_{k}(T(y ))| \rightarrow 0$, as $k\rightarrow\infty$, combining all and from the inequality $$|g_{k}(Tx_k)|\leq |g_{k}(T(y_{k}-y ))|+ |g_{k}(T(y ))|+ |g_{k}(Tz_{k})|\longrightarrow 0$$
	we conclude that $|g_{k}(Tx_k)|\longrightarrow 0$, this contradicts $|g_{k}(Tx_k)|> \varepsilon$. Hence,  $T(B_{E})$ is an L-weakly compact set of $F$, and this completes the proof.

\end{proof}
Note that a $\sigma$-uaw-compact (hence $\sigma$-un-compact) subset of a Banach lattice $E$ is not necessary limited. We observe that by considering the unit ball  $B_{\ell^{1}}$ of $\ell^{1}$ which is $\sigma$-un-compact, and hence $\sigma$-uaw-compact, by [\cite{article.4}, Theorem 7.5], but fails to be limited. However, we have the following result:

\begin{theorem} \label{uaw-limited}
	let $E$ and $F$ a pair of Banach lattices, such that $F$ has the DP* property. The following are equivalent:
	\begin{enumerate}
		\item $T:E\rightarrow F$ is disjoint weak$^{\star}$ Dunford-Pettis.
		\item  $T$ carries $\sigma$-uaw-compact subsets of $E$ into limited subsets of $F$
	\end{enumerate}

\end{theorem}
\begin{proof}
$1\Longrightarrow 2$ Let $A$ be a $\sigma$-uaw-compact subset of $E$, we need to show that $T(A)$ is limited subset of $F$. To this end, let  $f_{n} \overset{w^{\star}}{\longrightarrow} 0$ in $F'$, let $y_n\in T(A)$, that is, there exists $x_n\in A$ such that $y_n=T(x_n)$. Since $A$ is  $\sigma$-uaw-compact, there exists $(x_{n_k})$ a subsequence of $(x_n)$ and $x\in E$ such that $x_{n_k} \overset{uaw}{\longrightarrow} x$, we may assume that $x=0$. By Theorem  8.11 \cite{Thesis-ms}, there exists $(x_{n_{k'}})$ a subsequence of $(x_{n_k})$ a subsequence and a disjoint sequence $(d_{k'})$ satisfying $x_{n_{k'}}-d_{k'} \overset{w}{\longrightarrow}0$, therefore $T(x_{n_{k'}}-d_{k'} )\overset{w}{\longrightarrow}0$. From one hand, as $F$ has DP* property,  $f_{k'}(T(x_{n_{k'}}-d_{k'} )) \longrightarrow 0$. On the other hand, by hypothesis  $T$ is   disjoint weak$^{\star}$ Dunford-Pettis, that yields $f_{k'}(Td_{k'}) \longrightarrow 0$. Hence $f_{k'}(Tx_{n_{k'}}) \longrightarrow 0$. By the same argument, we showed that every subsequence of $(f_{n}(Tx_{n}))$ has a further subsequence which converges to zero, therefore $f_{n}(Tx_{n}) \longrightarrow 0$, which proves that $T(A)$ is limited subset .

 $2\Longrightarrow 1$ Let $(x_n)$ be a disjoint sequence of $E$ and $f_{n} \overset{w*}{\longrightarrow} 0$ in $F'$, we consider the subset $W=\{{x_{n},n\in \mathbb{N}\}} \cup\{0\} $, clearly $W$ is $\sigma$-uaw-compact subset of $E$, so by hypothesis $T(W)$ is limited in $F$. Therefore, $f_{n}(Tx_{n}) \longrightarrow 0$, this shows that $T$ is disjoint weak$^{\star}$ Dunford-Pettis.
	
\end{proof}

\begin{proposition}
	Let $E$ and $F$ be a pair of Banach lattices, such that $E$ is order continuous. The following are equivalent:
	\begin{enumerate}
		\item $T:E\rightarrow F$ is disjoint weak$^{\star}$ Dunford-Pettis.
		\item  $T$ carries $\sigma$-un-compact subsets of $E$ into limited subsets of $F$
	\end{enumerate}
	
\end{proposition}
\begin{proof}
$1\Longrightarrow 2$ Let $A$ be a $\sigma$-un-compact subset of $E$, we need to show that $T(A)$ is limited subset of $F$. Let consider,  $f_{n} \overset{w*}{\longrightarrow} 0$ in $F'$, and $y_n\in T(A)$. Since $A$ is $\sigma$-$un$-compact, there exist $(x_{n_k})$ a subsequence of $(x_n)$ and $x\in E$ such that $x_{n_k} \overset{un}{\longrightarrow} x$, we may assume that $x=0$. By [\cite{article.8}, Theorem 3.2] there exists $(x_{n_{k'}})$ a subsequence of $(x_{n_k})$  and a disjoint sequence $d_{k'}$ satisfying $\|x_{n_{k'}}-d_{k'} \|\longrightarrow0$. Consequently, we have $\|T(x_{n_{k'}}-d_{k'}) \|\longrightarrow0$, which implies 
 $f_{k'}(T(x_{n_{k'}}-d_{k'} )){\longrightarrow}0$. 
By hypothesis, $T$ is disjoint weak$^{\star}$ Dunford-Pettis, that yields $f_{k'}(Td_{k'}) \longrightarrow 0$. Hence $f_{k'}(Tx_{n_{k'}}) \longrightarrow 0$. By the same argument, we showed that every subsequence of $(f_{n}(Tx_{n}))$ has a further subsequence which converges to zero, therefore $f_{n}(Tx_{n})\longrightarrow 0$.

$2\Longrightarrow 1$ Let $(x_n)$ be a disjoint sequence of $E$, it follows from [\cite{article.7}, Lemma 2 and Theorem 4] that  $W=\{{x_{n},n\in \mathbb{N}\}} \cup\{0\} $, is $\sigma$-un-compact subset of $E$, by hypothesis $T(W)$ is limited in $F$. Therefore, for an arbitrary sequence $(f_{n})$ in $F'$ such that $f_{n} \overset{w^{\star}}{\longrightarrow} 0$, we obtain  $f_{n}(Tx_{n}) \longrightarrow 0$, which shows that $T$ is disjoint weak$^{\star}$ Dunford-Pettis.
	
\end{proof}

\begin{proposition} \label{characterization}
	Let $E$ and $F$ be a pair of Banach lattices, such that $E'$ is order continuous and let $T:E\rightarrow F$ be a positive operator. The following are equivalent:
	\begin{enumerate}
		\item $T$ is disjoint weak$^{\star}$ Dunford-Pettis.
		\item  For every disjoint weakly null sequence $(x_n)$ in $E^+$, and every weak* null
		sequence $(f_n)$ in $E'$, it follows that $f_{n}(Tx_{n})\longrightarrow 0$.
	\end{enumerate}	
\end{proposition}

\begin{proof}
	$(1)\Longrightarrow (2)$  is obvious.

	 For $(2)\Longrightarrow (1)$ Let $(x_n)$ be a disjoint sequence of $E$ and  $f_{n} \overset{w*}{\longrightarrow} 0$ in $F'$. Both  $(x_n^{+})$ and $(x_n^{-})$ are positive disjoint sequences. By Theorem 2.4.14 of \cite{book.2} $x_n^{+}\overset{w}{\longrightarrow} 0$ and $x_n^{-}\overset{w}{\longrightarrow} 0$. Hence, by hypothesis  we obtain $f_{n}(Tx_{n})=f_{n}(Tx_n^{+})- f_{n}(Tx_n^{-})\longrightarrow 0$.
\end{proof}

Let recall from \cite{Ratbi} that an operator $T$ from a Banach lattice $E$ to a Banach lattice $F$  called is almost weak* Dunford-Pettis, if $f_{n}(Tx_{n}) \longrightarrow 0$ for every weakly null sequence
$(x_n)$ in $E$ consisting of pairwise disjoint terms, and for every weak* null sequence $(f_n)$ in $F^{\prime}$ consisting of pairwise disjoint terms.

It easily observed that each disjoint weak$^{\star}$ Dunford-Pettis operator is almost weak* Dunford-Pettis. However the converse is not true in general. For instance, the identity operator $Id_{\ell^1}$ of $\ell^1$ is almost weak* Dunford-Pettis, but fails to be disjoint weak$^{\star}$ Dunford-Pettis.

As consequence of the Proposition above and [\cite{Ratbi},Theorem 2.3] we obtain the relationship between 
 both classes.
\begin{corollary}
	let $E$ and $F$ a pair of Banach lattices, such that $E'$ is order continuous and $F$ is $\sigma$-Dedekind complete. For every positive operator $T$ from $E$ into $F$, the following assertions are equivalent:
	\begin{enumerate}
		\item $T$ is  almost weak* Dunford-Pettis operator.
		\item  $T$ is disjoint weak$^{\star}$ Dunford-Pettis operator
	\end{enumerate}	
	
\end{corollary}

\begin{proof}
	$(2)\Longrightarrow (1)$ It is clear that each disjoint weak$^{\star}$ Dunford-Pettis operator is almost weak* Dunford-Pettis. 
    
 $(1)\Longrightarrow (2)$   If $T$ is almost weak* Dunford-Pettis, it follows from Theorem 2.3 \cite{Ratbi} that for every disjoint weakly null sequence $(x_n)$ in $E^+$, and every weak* null sequence $(f_n)$ in $E'$,  $f_{n}(Tx_{n})\longrightarrow 0$. We claim by Proposition \ref{characterization} that $T$ is a disjoint weak$^{\star}$ Dunford-Pettis operator.
\end{proof}	

\begin{corollary}
	For a   $\sigma$-Dedekind complete Banach lattice $E$ such that $E^{\prime}$ is a order continuous. Then,
	$E$ has the weak DP* property if and only if $E$ has the disjoint weak*  Dunford-Pettis property.
\end{corollary}


\begin{thebibliography}{9}
	\bibitem{book.1}  C.D. Aliprantis, O. Burkinshaw, {\it Positive Operators}, Springer, Berlin (2006).
	\bibitem{article.SDP} H. Carrión P. Galindo M. L. Lourenço {\it A stronger Dunford Pettis property} Studia Mathematica 2008.
	
	\bibitem{article.8} Y. Deng, M. O'Brien, V.G. Troitsky {\it Unbounded norm convergence in Banach lattices} Positivity {\bf21} (2017)
		
		
		
	\bibitem{article.DF} P.G. Dodds and D.H. Fremlin. {\it Compact operators on Banach lattices}. Israel J. Math. 34, 287-320 (1979).
		
	\bibitem{article.6} N. Erkursun-Ozcan, N. Anil Gezer, O. Zabeti, {\it Unbounded Absolutely Weak Dunford-Pettis Operators} Turk J Math, {\bf 43} (2019), 2731-2740.
		
		
	\bibitem{article.FF} A. El Kaddouri, S. Boumnidel, O. Aboutafail and K. Bouras, {\it The class of uaw w$^{*}$Dunford-Pettis Operators} Acta Math. Univ. Comenianae (2023), pp. 5564	
		
	\bibitem{FLO} J. Flores, and Ces\'ar Ruiz, \textit{Domination by positive Banach-Saks operators}, Studia Mathematica 173.2 (2006): 185-192.
	\bibitem{article.JJ}  J.X. Chen, Z.L. Chen, J.X. Guo, {\it Almost limited sets in Banach lattices} Journal of Mathematical Analysis and Applications
	Volume 412, Issue 1, 1 April 2014, Pages 547-553			\bibitem{article.4} M. Kandi\'{c}, M.A.A. Marabeh, V.G.Troitsky, {\it Unbounded norm topology in Banach lattices} J. Math. Anal. Appl. {\bf451} 259–279 (2017)
			
	\bibitem{book.2} P. Meyer-Nieberg, {\it Banach lattices} Universitext, Springer-Verlag, Berlin, 1991.

	\bibitem{article.dim} N.E. Ozcan, N.A. Gezer, O. Zabeti, {\it Dunford-Pettis and compact operators based on unbounded absolute weak convergence.} arXiv:1708.03970v7.


 

	
	

	\bibitem{article.5} M. Ouyang, Z. Chen, J. Chen, Z. Wang {\it$\sigma$-Unbounded Dunford-Pettis operators on Banach lattices}, 1909.07681v2, math.FA, (2019).	
	\bibitem{Ratbi} Abderrahman Retbi,{\it On the class of positive almost weak*
	Dunford-Pettis operators}, Comment. Math. Univ. Carolin., 347–354, 2015

    \bibitem{Thesis-ms} M.A. Taylor,{\it Unbounded convergences in vector lattices}, MSc. Thesis,
2018.
	\bibitem{zan}  A.C. Zaanen,{\it Riesz Spaces II}. North Holland Publishing Companyt, Amsterdam (1983).
	\bibitem{article.7} O. Zabeti {\it Unbounded absolute weak Convergence in Banach Lattices} Positivity, {\bf22(1)} (2018), pp. 501-505.	
\bibitem{article.usp} O. Zabeti, {\it Unbounded continuous operators and unbounded Banach-Saks property in Banach lattices}, Positivity 25, 1989–2001 (2021).	
	\end{thebibliography}
\end{document}